\documentclass[11pt]{article}

\usepackage{amsmath,amssymb,amsthm,mathrsfs}
\usepackage{hyperref}
\usepackage{geometry}
\usepackage{enumitem}
\usepackage{mathtools}
\usepackage{tikz}
\usetikzlibrary{positioning,decorations.pathreplacing,calc}
\newtheorem{theorem}{Theorem}[section]
\newtheorem{lemma}[theorem]{Lemma}
\newtheorem{proposition}[theorem]{Proposition}
\newtheorem{corollary}[theorem]{Corollary}
\newtheorem{problem}[theorem]{Problem}
\newtheorem{observation}[theorem]{Observation}
\newtheorem{example}[theorem]{Example}
\theoremstyle{definition}
\newtheorem{definition}[theorem]{Definition}
\theoremstyle{remark}
\newtheorem{remark}[theorem]{Remark}

\newcommand{\RR}{\mathbb{R}}

\newcommand{\dO}{\delta\Omega}
\newcommand{\pO}{\partial\Omega}
\newcommand{\Om}{\Omega}

\title{Effect of edge-stretching on Steklov eigenvalues \\and sharp Steklov eigenvalue bounds  on leaf--boundary trees}
\author{
    Jiangdong Ai\thanks{School of Mathematical Sciences and LPMC, Nankai University, Tianjin 300071, China. {\tt jd@nankai.edu.cn}. Funded by the National Natural Science Foundation of China (No.12522117, No.12401456), the Natural Science Foundation of Tianjin (No.24JCQNJC01960) and Fundamental and Interdisciplinary Disciplines Breakthrough Plan of the Ministry of Education of China (JYB2025XDXM207).}
    \hspace{2mm}
    Yizhe Ji\thanks{School of Mathematical Sciences, Nankai University, Tianjin 300071, China. {\tt jyz\_math@nankai.edu.cn}.}
    \hspace{2mm}
    Xiaopan Lian\thanks{Center for Combinatorics and LPMC, Nankai University, Tianjin 300071, China. {\tt{Lian@nankai.edu.cn.}}}
    \hspace{2mm}
    Kun Yang\thanks{School of Mathematical Sciences, Nankai University, Tianjin 300071, China. {\tt 2313878@mail.nankai.edu.cn}.}
}
\date{}

\begin{document}
\maketitle

\begin{abstract}
Let $T$ be a finite tree with leaf set $\dO$ as the boundary and let $\lambda_2$ be the first nontrivial Steklov eigenvalue.
Let $D$ and $\ell$ be the maximum vertex degree and the number of leaves, respectively. Motivated by the spectral influence of neck-stretching on Riemannian manifolds, we investigate a discrete counterpart--edge-stretching--and its effect on the Steklov eigenvalues of graphs. We prove that Steklov eigenvalues decrease monotonically under the edge--stretching operation. As a consequence, we prove that $\lambda_2\le D/\ell$, with equality if and only if $T$ is a star.
This fundamentally improves the constant in He--Hua's bound $\lambda_2\le 4(D-1)/\ell$ to the optimal value~$1$.

We also provide a closed-form diagonalization of the Steklov problem on level--regular trees,
yielding explicit eigenvalues and multiplicities.
In addition, we provide a general upper bound $\lambda_k\le \min\{1,\,16Dk/\ell\}$ for higher eigenvalues.
Systematic numerical experiments verify the sharp bound and provide evidence for the extremal conjecture of Lin--Zhao on balanced minimum--height trees.
\end{abstract}



\section{Introduction}\label{sec:intro}

Steklov~\cite{Stekloff1902} considered the problem of liquid sloshing and introduced Steklov eigenvalues and Steklov operators for bounded domains in Euclidean spaces. For a compact smooth orientable Riemannian manifold $\Om$ with boundary $\partial\Om$, the \emph{Dirichlet-to-Neumann operator} (abbreviated as DtN) $\mathcal{D}:C^\infty(\pO)\rightarrow C^\infty(\pO)$ is defined as a mapping from a boundary function $g\in C^\infty(\pO)$ to the normal derivative $\mathcal{D}f=\dfrac{\partial f}{\partial n}$ of its harmonic extension $f\in C^\infty(\Om)$. The eigenvalues of $\mathcal{D}$, known as \emph{Steklov eigenvalues}, form a discrete sequence
\[
0=\lambda_1(\Om,\pO)\le \lambda_2(\Om,\pO)\le \cdots\rightarrow\infty.
\]
These eigenvalues encode the interaction between the interior geometry of the manifold and the geometry of its boundary, and have been widely studied in spectral geometry. Recent developments on the Steklov problem are summarized in~\cite{CGGS2024,GirouardPolterovich2017}. 

A fundamental problem in spectral geometry is to analyze how geometric deformations of a manifold influence its spectral invariants. In the context of Steklov eigenvalues, an important class of deformations commonly referred to as \emph{neck stretching} involves attaching long cylindrical regions near the boundary. In this construction, a cylindrical region of the form $[0,L]\times\Sigma$ is inserted along a boundary component, and the spectral behavior is studied as the length $L$ increases. See Figure~\ref{fig:neckstretch}. Such geometric modifications alter the Dirichlet-to-Neumann operator and lead to changes in the Steklov spectrum, and they have been used in several works to investigate extremal eigenvalue problems and degenerating families of manifolds. See~\cite{GirouardHenrotLagace2021,Thibault2024} for more details.

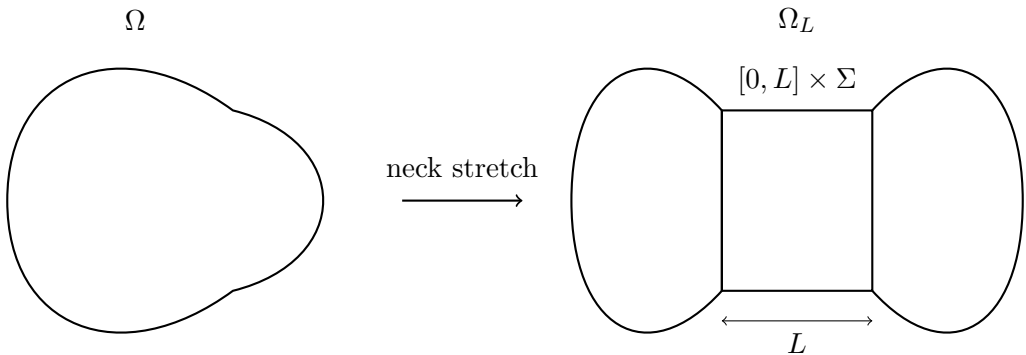
\begin{figure}[ht]\label{fig:neckstretch}
\centering
\begin{tikzpicture}[scale=1]

\draw[thick]
(-3-0.5,0)
.. controls (-3-0.5,1.7) and (-1.5-0.5,2.3) .. (0-0.5,1.2)
.. controls (1.6-0.5,0.8) and (1.6-0.5,-0.8) .. (0-0.5,-1.2)
.. controls (-1.5-0.5,-2.3) and (-3-0.5,-1.7) .. (-3-0.5,0);

\node at (-1.3-0.5,2.4) {$\Omega$};

\draw[->,thick] (1+0.75,0) -- (2.6+0.75,0);
\node at (1.8+0.75,0.45) {neck stretch};

\draw[thick]
(4,0)
.. controls (4,1.7) and (5,2.3) .. (6,1.2)
-- (6,-1.2)
.. controls (5,-2.3) and (4,-1.7) .. (4,0);

\draw[thick] (6,1.2) rectangle (8,-1.2);

\node at (7,1.6) {$[0,L]\times\Sigma$};

\draw[<->] (6,-1.6) -- (8,-1.6);
\node at (7,-1.9) {$L$};

\draw[thick]
(8,1.2)
.. controls (9,2.3) and (10,1.7) .. (10,0)
.. controls (10,-1.7) and (9,-2.3) .. (8,-1.2);

\node at (7,2.4) {$\Omega_L$};

\end{tikzpicture}

\caption{Neck stretching obtained by inserting a cylindrical region
$[0,L]\times\Sigma$.}
\end{figure}

The discrete analogue of the Steklov problem has recently attracted significant interest. Foundational work by Hua--Huang--Wang~\cite{HuaHuangWang2017}, Hassannezhad--Miclo~\cite{HassannezhadMiclo2020} and
He--Hua~\cite{HeHua2022Upper} established key spectral properties on graphs, while in~\cite{HeHua2022Flows} they obtained crucial upper bound estimates and proved the monotonicity of the first Steklov eigenvalue under specific tree flows.
Lin--Zhao~\cite{LinZhaoPlanar} extended these results to planar graphs and block graphs. There are some other results with different graph parameters in~\cite{ChenShiZhang2025,LinLYZ2026, LinZhaoConnectivity,LinZhaoTrees,Perrin2019}.

A graph $G=(V,E)$ is a tuple of the vertex set $V$ and the edge set $E$.
In this paper, all graphs are simple.
The boundary of the graph, denoted by $\dO$, is chosen as a non-empty subset of~$V$,
and we use $\Om=V\setminus \dO$ for the set of remaining vertices.
For a real function $f\in\RR^V$, we define the Laplacian operator $\Delta f$ by
\begin{equation}\label{eq:lap}
(\Delta f)(x)=\sum_{(x,y)\in E}\bigl(f(x)-f(y)\bigr),
\end{equation}
and the outward derivative operator $\frac{\partial f}{\partial n}$ by
\begin{equation}\label{eq:normal}
\frac{\partial f}{\partial n}(x)=\sum_{\substack{y\in V\\(x,y)\in E}}\bigl(f(x)-f(y)\bigr),\qquad x\in\dO.
\end{equation}
The Steklov problem solves the following equations for a real function $f\in\RR^V$ and a real number $\lambda\in\RR$:
\begin{equation}\label{eq:steklov}
\begin{cases}
\Delta f(x)=0, & x\in\Om,\\[2pt]
\dfrac{\partial f}{\partial n}(x)=\lambda f(x), & x\in\dO.
\end{cases}
\end{equation}
There are $|\dO|$ Steklov eigenvalues, arranged as
\[
0=\lambda_1(G,\dO)\le \lambda_2(G,\dO)\le \cdots\le \lambda_{|\dO|}(G,\dO).
\]
We abbreviate $\lambda_i(G,\dO)$ as $\lambda_i$ when the context is clear.
The Steklov eigenvalues can be characterized by the Rayleigh quotient:
\begin{equation}\label{eq:minmax}
\lambda_k=\min_{\substack{W\subseteq\RR^{\dO}\\\dim W=k}}\ \max_{0\neq g\in W}\ \mathcal{R}_{G,\dO}(g),\qquad
\mathcal{R}_{G,\dO}(g)=\frac{\mathcal{E}_G(\widehat{g})}{\sum_{x\in\dO}g(x)^2},
\end{equation}
where $\widehat{g}$ denotes the harmonic extension of $g$ to~$V$ (see Section~\ref{sec:proofs} for definitions). 

A tree is a connected graph without cycles.
The leaf set is the set of vertices of degree one.
Throughout this paper, we consider trees with the leaf set as the boundary.

In view of the geometric deformation described above, it is natural to search for graph operations that mimic the effect of stretching regions of a manifold. The following definition can be regarded as a natural discrete analogue of neck-stretching. 

\begin{definition}\label{def:stretch}
Let $(G,\dO)$ be a graph with boundary and let $e=(u,v)\in E$.
For an integer $L\ge 1$, the \emph{$L$--stretch} of $e$ is the graph $G^{(L)}$ obtained by replacing $uv$ with a path $u=w_0\sim w_1\sim \cdots\sim w_L=v$ of length $L$, where $w_1,\dots,w_{L-1}$ are new interior vertices.
The boundary set is unchanged.
See Figure~\ref{fig:stretch}.
\end{definition}

\begin{figure}[ht]
\centering
\begin{tikzpicture}[
    vertex/.style={circle, draw, fill=black, inner sep=1.5pt},
    bvertex/.style={circle, draw, fill=white, inner sep=1.5pt},
    newvertex/.style={circle, draw, fill=gray!50, inner sep=1.5pt},
    every edge/.style={draw, thick},
    >=stealth
]
\node at (-1.5-0.65,0.8) {\textbf{Original}};
\node[vertex, label=below:{$u$}] (u) at (-2-0.65,0) {};
\node[vertex, label=below:{$v$}] (v) at (-1-0.65,0) {};
\node[bvertex] (a) at (-2.7-0.65,0.5) {};
\node[bvertex] (b) at (-2.7-0.65,-0.5) {};
\node[bvertex] (c) at (-0.3-0.65,0.5) {};
\node[bvertex] (d) at (-0.3-0.65,-0.5) {};
\draw (a)--(u); \draw (b)--(u);
\draw (c)--(v); \draw (d)--(v);
\draw[very thick, red] (u)--(v);

\draw[->, thick] (0.3-0.5,0) -- (1.3,0) node[midway, above] {$3$-stretch};

\node at (3.5+0.65,0.8) {\textbf{After stretching}};
\node[vertex, label=below:{$u$}] (u2) at (2+0.65,0) {};
\node[newvertex, label=below:{$w_1$}] (w1) at (3+0.65,0) {};
\node[newvertex, label=below:{$w_2$}] (w2) at (4+0.65,0) {};
\node[vertex, label=below:{$v$}] (v2) at (5+0.65,0) {};
\node[bvertex] (a2) at (1.3+0.65,0.5) {};
\node[bvertex] (b2) at (1.3+0.65,-0.5) {};
\node[bvertex] (c2) at (5.7+0.65,0.5) {};
\node[bvertex] (d2) at (5.7+0.65,-0.5) {};
\draw (a2)--(u2); \draw (b2)--(u2);
\draw (c2)--(v2); \draw (d2)--(v2);
\draw[very thick, red] (u2)--(w1)--(w2)--(v2);
\end{tikzpicture}
\caption{The $3$--stretch of an edge $e=(u,v)$: the edge is replaced with a path of length~$3$ through new interior vertices $w_1,w_2$ (gray).}
\label{fig:stretch}
\end{figure}
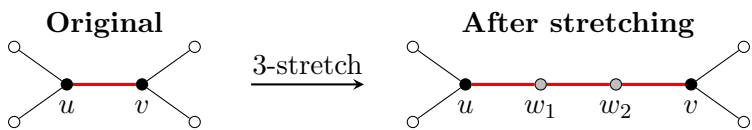

From the geometric perspective, this operation increases the effective distance between vertices in the graph and can be viewed as a combinatorial analogue of inserting a long cylindrical region in the continuous setting.

The first result concerns the monotonicity of the Steklov spectrum under edge stretching.

\begin{theorem}\label{thm:stretch}
Let $G^{(L)}$ be obtained from $(G,\dO)$ by $L$--stretching a single edge.
Then for every $k=1,\dots,|\dO|$,
\begin{equation}\label{eq:stretch}
\lambda_k(G^{(L)},\dO)\le \lambda_k(G,\dO).
\end{equation}
\end{theorem}

\begin{corollary}\label{cor:deg2}
Let $T$ be a tree with leaf boundary $\dO$.
If $x\in\Om$ is an interior vertex with $\deg(x)=2$, let $T_{\mathrm{ctr}}$ be obtained by contracting $x$.
Then $\lambda_k(T,\dO)\le \lambda_k(T_{\mathrm{ctr}},\dO)$ for every $k$.
In particular, any tree maximizing $\lambda_2$ within a family of trees with prescribed leaf number and degree bound may be assumed to have no interior degree--two vertex.
\end{corollary}

By Theorem~\ref{thm:stretch} and Corollary~\ref{cor:deg2}, we obtain a sharp upper bound for the first Steklov eigenvalue, which is the second main result of this paper.

\begin{theorem}\label{thm:lambda2sharp}
Let $T=(V,E)$ be a tree with leaf boundary $\dO$.
Let $D\ge 2$ be the maximum degree of $T$ and let $\ell=|\dO|$ be the number of leaves.
Then
\begin{equation}\label{eq:mainbound_intro}
\lambda_2\le \frac{D}{\ell}.
\end{equation}
Moreover, equality holds if and only if $T$ is a star.
\end{theorem}

Compare~\eqref{eq:mainbound_intro} with He--Hua's bound $\lambda_2\le 4(D-1)/\ell$ \cite[Theorem~1.1]{HeHua2022Upper} and Lin--Zhao's bound $\lambda_2\le 8D/\ell$ for planar graphs \cite[Theorem~1.1]{LinZhaoPlanar}.
Theorem~\ref{thm:lambda2sharp} improves the constant from $4$ to the optimal value $1$ for trees.

The bound $\lambda_2\le 4(D-1)/\ell$ of He--Hua~\cite{HeHua2022Upper} is derived via an isoperimetric (Cheeger--type) inequality for the Steklov problem.
Such inequalities involve a squaring step: one bounds $\lambda_2$ by the Cheeger ratio and then estimates the Cheeger ratio separately, which introduces a multiplicative constant.
A key novelty of our approach is that it bypasses the Cheeger inequality entirely, thereby avoiding the multiplicative constants inherent in Cheeger-type estimates. Instead, we work directly with the Rayleigh quotient at a leaf centroid. By constructing an explicit two-value test function whose energy can be optimized in closed form (Lemma~\ref{lem:vertex2value}), we achieve the optimal constant.
The key observation is that on a tree, removing a single vertex splits the boundary into components whose sizes are controlled by the centroid property; combining this with the convexity of the energy functional at a vertex yields the sharp constant without any intermediate inequality.

The third result gives explicit spectra for level--regular trees, which provide a large family of test candidates with closed formulas.

\begin{definition}\label{def:levelregular}
Fix $h\ge 1$ and integers $m_0,\dots,m_{h-1}\ge 1$.
A rooted tree $T(\mathbf{m})$ of height $h$ is \emph{level--regular with branching sequence $\mathbf{m}=(m_0,\dots,m_{h-1})$}
if every vertex at depth $t$ has exactly $m_t$ children.
Its leaves are precisely the vertices at depth $h$, and we take $\dO$ to be the leaf set.
See Figure~\ref{fig:levelreg}.
\end{definition}

\begin{figure}[ht]
\centering
\begin{tikzpicture}\begin{scope}[
    ivert/.style={circle, draw, fill=black, inner sep=1.5pt},
    lvert/.style={circle, draw, fill=white, inner sep=2pt},
    level distance=12mm,
    level 1/.style={sibling distance=20mm},
    level 2/.style={sibling distance=10mm},
    edge from parent/.style={draw,thick}
]
\node at (0,0.7) {$\mathbf{m}=(3,2)$, $h=2$};
\node[ivert, label=left:{\footnotesize $d\!=\!0$}] (r) at (0,0) {}
    child { node[ivert] {}
        child { node[lvert] {} }
        child { node[lvert] {} }
    }
    child { node[ivert] {}
        child { node[lvert] {} }
        child { node[lvert] {} }
    }
    child { node[ivert] {}
        child { node[lvert] {} }
        child { node[lvert] {} }
    };
\node[anchor=west] at (2.1,-1.2) {\footnotesize $d\!=\!1$};
\node[anchor=west] at (2.5,-2.4) {\footnotesize $d\!=\!2$\;(leaves)};\end{scope}
\begin{scope}[
    ivert/.style={circle, draw, fill=black, inner sep=1.5pt},
    lvert/.style={circle, draw, fill=white, inner sep=2pt},
    level distance=12mm,
    level 1/.style={sibling distance=30mm},
    level 2/.style={sibling distance=10mm},
    edge from parent/.style={draw,thick}
]
\node at (9,0.7) {$\mathbf{m}=(2,3)$, $h=2$};
\node[ivert] (r2) at (9,0) {}
    child { node[ivert] {}
        child { node[lvert] {} }
        child { node[lvert] {} }
        child { node[lvert] {} }
    }
    child { node[ivert] {}
        child { node[lvert] {} }
        child { node[lvert] {} }
        child { node[lvert] {} }
    };\end{scope}
\end{tikzpicture}
\caption{Level--regular trees with non-constant branching sequences. Left: $\mathbf{m}=(3,2)$, $6$ leaves. Right: $\mathbf{m}=(2,3)$, $6$ leaves. Both have $6$ leaves but different spectra (see Example~\ref{ex:nonconstant}).}
\label{fig:levelreg}
\end{figure}
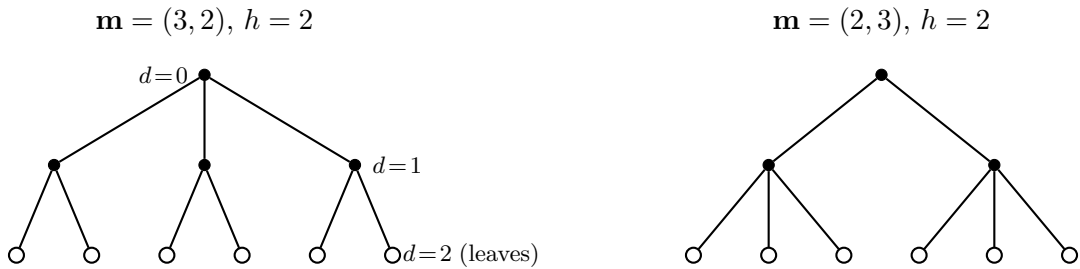

For $T(\mathbf{m})$ define, for $j=1,\dots,h$,
\begin{equation}\label{eq:muj}
\mu_j(\mathbf{m}):=\frac{1}{1+m_{h-1}+m_{h-2}m_{h-1}+\cdots+\prod_{t=h-j+1}^{h-1}m_t},
\end{equation}
with the convention $\mu_1(\mathbf{m})=1$.
Also let $N_d=\prod_{t=0}^{d-1}m_t$ denote the number of vertices at depth $d$, with $N_0=1$.

\begin{theorem}\label{thm:levelspec}
Let $T(\mathbf{m})$ be level--regular of height $h$ with leaf boundary $\dO$.
Then the Steklov eigenvalues are:
\begin{itemize}[leftmargin=2em]
\item $\lambda_1=0$ with multiplicity $1$;
\item for each $j=1,\dots,h$, the eigenvalue $\mu_j(\mathbf{m})$ has multiplicity $(m_{h-j}-1)N_{h-j}$.
\end{itemize}
\end{theorem}

\begin{corollary}\label{cor:mary}
For the complete $m$--ary tree $T_{m,h}$, where $m_t\equiv m$ and $m\ge 2$, the spectrum is $0$, $\mu_j=(m-1)/(m^j-1)$ for $j=1,\dots,h$, where $\mu_j$ has multiplicity $(m-1)m^{h-j}$.
In particular,
\begin{equation}\label{eq:mary}
\lambda_2(T_{m,h},\dO)=\frac{m-1}{m^h-1}.
\end{equation}
\end{corollary}

Corollary~\ref{cor:mary} shows that for the complete $(D-1)$--ary tree, $\lambda_2=(D-2)/((D-1)^h-1)\sim (D-2)/\ell$.
Comparing with Theorem~\ref{thm:lambda2sharp}, the ratio $\lambda_2/(D/\ell)\to 1$ as $D\to\infty$, so the bound is asymptotically tight at the natural scale $D/\ell$.

Finally, we give an upper bound for higher eigenvalues.

\begin{theorem}\label{thm:lambdak}
Let $T$ be a tree with leaf boundary $\dO$, $|\dO|=\ell$, and maximum degree at most $D$.
Then for every $k=1,\dots,\ell$,
\begin{equation}\label{eq:lambdak}
\lambda_k(T,\dO)\le \min\!\left\{1,\;\frac{16Dk}{\ell}\right\}.
\end{equation}
\end{theorem}

The paper is organized as follows.
In Section~\ref{sec:proofs}, we provide the proofs.
In Section~\ref{sec:numerics}, we present numerical experiments.
In Section~\ref{sec:conclude}, we conclude with remarks and open problems.

\section{Proofs}\label{sec:proofs}

\subsection{Preliminaries}\label{subsec:prelim}
Throughout the general graph discussion, we assume that no two boundary vertices are adjacent.

For $f\in\RR^V$, the Dirichlet energy is
\[
\mathcal{E}_G(f):=\sum_{(x,y)\in E}\bigl(f(x)-f(y)\bigr)^2.
\]
If $g\in\RR^{\dO}$ is boundary data, its harmonic extension $\widehat{g}$ is the unique function on $V$ with $\widehat{g}|_{\dO}=g$ and $\Delta\widehat{g}=0$ on $\Om$.

\begin{lemma}[Harmonic extension minimizes energy]\label{lem:harmmin}
Fix $g\in\RR^{\dO}$.
Among all $f\in\RR^V$ with $f|_{\dO}=g$, the energy $\mathcal{E}_G(f)$ is minimized by the harmonic extension $\widehat{g}$.
\end{lemma}

\begin{proof}
Write $f=\widehat{g}+h$ with $h|_{\dO}=0$, expand the energy, and use $\Delta\widehat{g}=0$ on $\Om$.
\end{proof}

The Steklov Rayleigh quotient of nonzero boundary data $g\in\RR^{\dO}$ is
\begin{equation}\label{eq:rayleigh}
\mathcal{R}_{G,\dO}(g):=\frac{\mathcal{E}_G(\widehat{g})}{\sum_{x\in\dO}g(x)^2}.
\end{equation}
The min--max principle then gives, for $k=1,\dots,|\dO|$,
\begin{equation}\label{eq:minmax_formal}
\lambda_k(G,\dO)=\min_{\substack{W\subseteq\RR^{\dO}\\\dim W=k}}\ \max_{0\neq g\in W}\ \mathcal{R}_{G,\dO}(g).
\end{equation}
In several proofs below it is convenient to construct vertex functions $f\in\RR^V$ with prescribed boundary values $g=f|_{\dO}$;
by Lemma~\ref{lem:harmmin}, $\mathcal{R}_{G,\dO}(g)\le \mathcal{E}_G(f)/\sum_{\dO}g^2$ for any such extension $f$, so any extension provides a valid upper bound for the Rayleigh quotient.

Let $L$ be the Laplacian matrix of $G$ ordered so that boundary vertices come first:
\[
L=\begin{pmatrix}
L_{\dO\dO} & L_{\dO\Om}\\[2pt]
L_{\Om\dO} & L_{\Om\Om}
\end{pmatrix}.
\]

\begin{proposition}[Dirichlet--to--Neumann operator]\label{prop:schur}
Assume $G$ is connected and $\dO\neq\varnothing$.
Then $L_{\Om\Om}$ is invertible and the Dirichlet--to--Neumann operator $\Lambda:\RR^{\dO}\to\RR^{\dO}$ is the symmetric matrix
\begin{equation}\label{eq:schur}
\Lambda=L_{\dO\dO}-L_{\dO\Om}L_{\Om\Om}^{-1}L_{\Om\dO}.
\end{equation}
The Steklov eigenvalues of $(G,\dO)$ coincide with the eigenvalues of $\Lambda$.
\end{proposition}

\begin{proof}
This is the standard Schur complement formula (also known as Kron reduction).
\end{proof}

\begin{corollary}\label{cor:leq1}
Let $T$ be a tree with leaves as boundary $\dO$.
Then $\lambda_k(T,\dO)\le 1$ for every $k$.
\end{corollary}

\begin{proof}
For a tree with leaf boundary, no two boundary vertices are adjacent and each has degree~$1$.
Hence $L_{\dO\dO}=I$, and $\Lambda=I-L_{\dO\Om}L_{\Om\Om}^{-1}L_{\Om\dO}$, where the subtracted term is positive semidefinite.
Thus $0\preceq\Lambda\preceq I$.
\end{proof}

\subsection{Edge stretching}\label{subsec:stretch}

\begin{lemma}\label{lem:path}
Let $a,b\in\RR$ and $L\ge 1$.
Among all sequences $(t_0,\dots,t_L)$ with $t_0=a$ and $t_L=b$,
\[
\sum_{i=0}^{L-1}(t_{i+1}-t_i)^2
\]
is minimized by linear interpolation $t_i=a+\frac{i}{L}(b-a)$, and the minimum equals $(a-b)^2/L$.
\end{lemma}

\begin{proof}
Let $d_i=t_{i+1}-t_i$, so $\sum_i d_i=b-a$.
By Cauchy--Schwarz, $\sum_i d_i^2\ge \frac{1}{L}\bigl(\sum_i d_i\bigr)^2=(a-b)^2/L$, with equality if and only if all $d_i$ are equal.
\end{proof}

\begin{proof}[Proof of Theorem~\ref{thm:stretch}]
Fix any $k$--dimensional subspace $W\subseteq\RR^{\dO}$.
For $g\in W$, let $\widehat{g}$ be the harmonic extension on $G$.
Define an extension $\widetilde{g}$ on $G^{(L)}$ by:
(i) $\widetilde{g}=\widehat{g}$ on the original vertices, and
(ii) on the stretched path, interpolate linearly between $\widehat{g}(u)$ and $\widehat{g}(v)$.
All energy terms coincide except on the stretched edge, where Lemma~\ref{lem:path} gives an energy factor $1/L$.
Thus $\mathcal{E}_{G^{(L)}}(\widetilde{g})\le \mathcal{E}_G(\widehat{g})$.
By Lemma~\ref{lem:harmmin}, the harmonic extension on $G^{(L)}$ has energy at most $\mathcal{E}_{G^{(L)}}(\widetilde{g})$,
so $\mathcal{R}_{G^{(L)},\dO}(g)\le \mathcal{R}_{G,\dO}(g)$ for all $g\in W$.
Taking $\max_{0\neq g\in W}$ and then $\min_{\dim W=k}$ in~\eqref{eq:minmax_formal} yields the theorem.
\end{proof}

\begin{remark}[Comparison with existing monotonicity results]\label{rem:mono_compare}
He--Hua~\cite{HeHua2022Flows} proved that $\lambda_2$ is monotone non-decreasing under taking subtrees (with leaf boundary),
and Yu--Yu~\cite{YuYu2024} extended this to all $\lambda_k$ in the more general setting of graphs with combs.
Both results concern \emph{edge deletion}: removing an edge from a graph decreases the Rayleigh quotient.
Edge stretching is a different operation---it does not remove the edge but \emph{subdivides} it by inserting new interior vertices.
Theorem~\ref{thm:stretch} is complementary to these results: subdivision increases effective resistance along the edge (from $1$ to $L$ in the resistor--network picture), decreasing the Steklov eigenvalues without altering the boundary set or disconnecting the graph.
\end{remark}

\begin{proposition}[Strict decrease for simple eigenvalues]\label{prop:strict}
Suppose $\lambda_k(G,\dO)>\lambda_{k-1}(G,\dO)$.
Let $\phi_k$ be a Steklov eigenfunction with eigenvalue $\lambda_k$, normalized by $\|\phi_k\|_{\ell^2(\dO)}=1$.
If the stretched edge $e=(u,v)$ satisfies $\phi_k(u)\neq\phi_k(v)$, then for every $L\ge 2$,
\[
\lambda_k(G^{(L)},\dO)<\lambda_k(G,\dO).
\]
\end{proposition}

\begin{proof}
The proof of Theorem~\ref{thm:stretch} shows $\Lambda^{(L)}\preceq\Lambda$ in Loewner order.
Let $e_1,\dots,e_{|\dO|}$ be an orthonormal eigenbasis of $\Lambda$ with eigenvalues $\lambda_1\le\cdots\le\lambda_{|\dO|}$ and $e_k=\phi_k|_{\dO}$.
Since $\lambda_k>\lambda_{k-1}$, the vector $e_k$ is the unique (up to sign) unit vector maximizing the Rayleigh quotient on $F:=\mathrm{span}\{e_1,\dots,e_k\}$.
By Lemma~\ref{lem:path}, $e_k^\top\Lambda^{(L)}e_k < e_k^\top\Lambda e_k = \lambda_k$.
For any $g\in F$ with $g^\top\Lambda^{(L)}g=\lambda_k\|g\|^2$, we would need $g^\top\Lambda g=\lambda_k\|g\|^2$, forcing $g$ to be a multiple of $e_k$, which is impossible.
Hence $\max_{0\neq g\in F}g^\top\Lambda^{(L)}g/\|g\|^2<\lambda_k$, and the min--max principle concludes the proof.
\end{proof}

\begin{proof}[Proof of Corollary~\ref{cor:deg2}]
$T$ is obtained from $T_{\mathrm{ctr}}$ by a $2$--stretch of one edge. Apply Theorem~\ref{thm:stretch}.
\end{proof}

\subsection{Explicit spectra for level--regular trees}\label{subsec:levelspec}

We first need a flux computation on a symmetric branch.

\begin{lemma}\label{lem:flux}
Fix $j\ge 1$ and integers $b_1,\dots,b_{j-1}\ge 1$.
Consider a rooted level--regular tree of height $j-1$ in which vertices at distance $t$ from the root have exactly $b_{t+1}$ children ($t=0,\dots,j-2$), with leaves forming the boundary.
Attach an extra vertex $p$ to the root.
Prescribe $g\equiv 1$ on all leaves and $g(p)=0$.
Let $f$ be the harmonic extension of $g$.
Then for any boundary leaf $x$,
\begin{equation}\label{eq:rho}
\frac{\partial f}{\partial n}(x)
=\rho(b_1,\dots,b_{j-1})
:=\frac{1}{1+b_{j-1}+b_{j-2}b_{j-1}+\cdots+b_1 b_2\cdots b_{j-1}}.
\end{equation}
\end{lemma}

\begin{proof}
By symmetry, values depend only on the distance from $p$.
Let $a_t$ be the value at distance $t$, so $a_0=0$ and $a_j=1$.
Harmonicity at level $t$ ($1\le t\le j-1$) yields $a_t-a_{t-1}=b_t(a_{t+1}-a_t)$.
Setting $\delta_t:=a_t-a_{t-1}$ ($1\le t\le j$) gives $\delta_t=b_t\delta_{t+1}$, hence $\delta_t=\delta_j\cdot b_t b_{t+1}\cdots b_{j-1}$.
Since $1=a_j-a_0=\sum_{t=1}^j\delta_t$, we obtain $\delta_j=\rho(b_1,\dots,b_{j-1})$.
Finally, $\frac{\partial f}{\partial n}(x)=a_j-a_{j-1}=\delta_j$.
\end{proof}

\begin{proof}[Proof of Theorem~\ref{thm:levelspec}]
Fix $j$ and let $v$ be any vertex at depth $h-j$.
It has $m_{h-j}$ children; the boundary leaves descending from each child form a block.
Choose coefficients $\alpha=(\alpha_1,\dots,\alpha_{m_{h-j}})$ with $\sum_i\alpha_i=0$.
Define boundary data $g$ that equals $\alpha_i$ on the $i$--th child block of $v$ and equals $0$ on all other leaves.
Let $f$ be the harmonic extension.

\smallskip\noindent
\emph{Step~1: $f(v)=0$.}
The map $\alpha\mapsto f(v)$ is linear in $\alpha=(\alpha_1,\dots,\alpha_{m_{h-j}})$.
Since the child subtrees of $v$ are isomorphic, permuting them leaves $v$ fixed, so by symmetry $f(v)$ depends on $\alpha$ only through a multiple of $\sum_i\alpha_i$.
Because $\sum_i\alpha_i=0$, we obtain $f(v)=0$.

\smallskip\noindent
\emph{Step~2: Steklov boundary condition.}
On each child subtree under $v$, we are in the setting of Lemma~\ref{lem:flux} with parent value $0$ and leaf value $\alpha_i$.
Hence $\frac{\partial f}{\partial n}(x)=\mu_j(\mathbf{m})\alpha_i=\mu_j(\mathbf{m})f(x)$.

\smallskip\noindent
\emph{Step~3: Multiplicity.}
For fixed $v$, the space $\{\alpha:\sum_i\alpha_i=0\}$ has dimension $m_{h-j}-1$.
Different vertices at depth $h-j$ have disjoint descendant leaf sets, yielding independent eigenfunctions.
The multiplicity is $(m_{h-j}-1)N_{h-j}$.
Summing over $j$:
\[
1+\sum_{j=1}^h(m_{h-j}-1)N_{h-j}
=1+\sum_{d=0}^{h-1}(m_d-1)N_d
=1+(N_h-1)=N_h=|\dO|,
\]
so all eigenvalues are accounted for.
\end{proof}

\begin{proof}[Proof of Corollary~\ref{cor:mary}]
For $m_t\equiv m$, the denominator in~\eqref{eq:muj} is $1+m+\cdots+m^{j-1}=(m^j-1)/(m-1)$.
\end{proof}

\begin{example}[Non-constant branching sequences]\label{ex:nonconstant}
The closed formula~\eqref{eq:muj} applies to arbitrary branching sequences, not just constant ones.
Consider $\mathbf{m}=(3,2)$ with $h=2$ (Figure~\ref{fig:levelreg}, left).
Then $|\dO|=6$, and the eigenvalues are $\mu_1=1$ with multiplicity $(2-1)\cdot 3=3$, and $\mu_2=1/(1+2)=1/3$ with multiplicity $(3-1)\cdot 1=2$, together with $\lambda_1=0$.
For $\mathbf{m}=(2,3)$ with $h=2$ (Figure~\ref{fig:levelreg}, right), also $|\dO|=6$, but the eigenvalues are $\mu_1=1$ with multiplicity $(3-1)\cdot 2=4$ and $\mu_2=1/(1+3)=1/4$ with multiplicity $(2-1)\cdot 1=1$.
Thus $\lambda_2(T(3,2))=1/3>1/4=\lambda_2(T(2,3))$: placing more branching closer to the leaves increases $\lambda_2$.
\end{example}

\subsection{Sharp upper bound for \texorpdfstring{$\lambda_2$}{lambda\_2}}\label{subsec:sharp}

\begin{lemma}\label{lem:cut}
Let $(G,\dO)$ be connected and let $S\subseteq V$.
Write $p:=|\dO\cap S|/|\dO|$ and assume $0<p<1$.
Then
\begin{equation}\label{eq:cut}
\lambda_2(G,\dO)\le \frac{|E(S,V\setminus S)|}{p(1-p)\,|\dO|}.
\end{equation}
\end{lemma}

\begin{proof}
Define $f\equiv(1-p)$ on $S$ and $f\equiv -p$ on $V\setminus S$.
Then $\sum_{\dO}f=0$, $\sum_{\dO}f^2=p(1-p)|\dO|$, and $\mathcal{E}_G(f)=|E(S,V\setminus S)|$.
Harmonic extension can only decrease energy, so the Rayleigh quotient is at most the computed value.
\end{proof}

\begin{lemma}\label{lem:centroid}
Let $T$ be a tree and $\dO$ its set of leaves.
There exists a vertex $v$ such that each component $C$ of $T\setminus\{v\}$ satisfies $|\dO\cap C|\le |\dO|/2$.
\end{lemma}

\begin{proof}
Assign weight $1$ to each leaf and $0$ to each interior vertex.
A standard weighted centroid argument for trees applies.
\end{proof}

\begin{lemma}\label{lem:vertex2value}
Let $T$ be a tree with leaf boundary $\dO$ and $|\dO|=\ell$.
Fix a vertex $v$ of degree $r\ge 2$ and let $C_1,\dots,C_r$ be the components of $T\setminus\{v\}$.
Write $b_i:=|\dO\cap C_i|$ and fix an index $i_0$.
Set $s:=b_{i_0}$ and assume $0<s\le\ell/2$.
Then
\begin{equation}\label{eq:vertex2value}
\lambda_2(T,\dO)\le \frac{r-1}{r}\cdot\frac{\ell}{s(\ell-s)}.
\end{equation}
\end{lemma}

\begin{proof}
Let $\alpha:=s/(\ell-s)$ and define the boundary data $g\in\RR^{\dO}$ by $g\equiv 1$ on $\dO\cap C_{i_0}$ and $g\equiv -\alpha$ on $\dO\setminus C_{i_0}$.
Then $\sum_{\dO}g=s-(\ell-s)\alpha=0$, so $g\perp\mathbf{1}$.

Extend $g$ to $f\in\RR^{V(T)}$ by setting $f\equiv 1$ on $C_{i_0}$, $f\equiv -\alpha$ on each $C_i$ for $i\neq i_0$, and choosing $f(v)$ to minimize the energy across the $r$ edges incident to $v$.
A direct minimization gives $f(v)=\frac{1-(r-1)\alpha}{r}$ and
\[
\mathcal{E}_T(f)=(f(v)-1)^2+(r-1)(f(v)+\alpha)^2=\frac{r-1}{r}\,(1+\alpha)^2.
\]
Moreover,
\[
\sum_{x\in\dO}g(x)^2=s+(\ell-s)\alpha^2=\frac{s\ell}{\ell-s}.
\]
By Lemma~\ref{lem:harmmin}, the harmonic extension has energy at most $\mathcal{E}_T(f)$.
Using $1+\alpha=\ell/(\ell-s)$, we obtain
\[
\mathcal{R}_{T,\dO}(g)\le \frac{\frac{r-1}{r}\cdot\frac{\ell^2}{(\ell-s)^2}}{\frac{s\ell}{\ell-s}}
=\frac{r-1}{r}\cdot\frac{\ell}{s(\ell-s)}.
\]
By the min--max principle, $\lambda_2\le \mathcal{R}_{T,\dO}(g)$, proving~\eqref{eq:vertex2value}.
\end{proof}

\begin{proof}[Proof of Theorem~\ref{thm:lambda2sharp}]
Let $v$ be a leaf centroid as in Lemma~\ref{lem:centroid}.
Let $C_1,\dots,C_r$ be the components of $T\setminus\{v\}$, so $r=\deg(v)\le D$.
Set $b_i:=|\dO\cap C_i|$.
By the centroid property, each $b_i\le\ell/2$.
Let $i_0$ maximize $b_i$ and write $s:=b_{i_0}$.
By pigeonhole, $s\ge\ell/r$.

Apply Lemma~\ref{lem:vertex2value} at $v$ and $C_{i_0}$ to obtain
$$
\lambda_2\le \frac{r-1}{r}\cdot\frac{\ell}{s(\ell-s)}.
$$
Since $s\le\ell/2$, the function $x\mapsto\ell/(x(\ell-x))$ is decreasing on $(0,\ell/2]$, hence
$$
\frac{\ell}{s(\ell-s)}\le \frac{\ell}{(\ell/r)(\ell-\ell/r)}=\frac{r^2}{\ell(r-1)}.
$$
Therefore $\lambda_2\le \frac{r-1}{r}\cdot\frac{r^2}{\ell(r-1)}=\frac{r}{\ell}\le\frac{D}{\ell}$.

For the equality statement, equality throughout forces
$
r=D, s=\frac{\ell}{r},
$
and the harmonic extension must have the same energy as the piecewise--constant extension used in
Lemma~\ref{lem:vertex2value}. Hence the function $f$ constructed in Lemma~\ref{lem:vertex2value}, which is constant on each component
$C_i$, must itself be harmonic on $\Om$.

We claim that each component $C_i$ contains no interior vertex. Indeed, suppose that some $C_i$
contains an interior vertex. Since $C_i$ is a component of $T\setminus\{v\}$, there is a unique vertex
$w\in C_i$ adjacent to $v$. This vertex $w$ lies in $\Om$, because all boundary vertices are leaves
and $w$ has the neighbor $v$ together with at least one neighbor inside $C_i$. By construction, $f$
is constant on $C_i$, so every neighbor of $w$ inside $C_i$ has the same value as $f(w)$. On the other
hand, the value at $v$ is different from the constant value on $C_i$. Therefore
$$
\sum_{w'\sim w}\bigl(f(w)-f(w')\bigr)=f(w)-f(v)\neq 0,
$$
so $f$ is not harmonic at $w$, a contradiction.

Thus every $C_i$ contains no interior vertex, and hence every $C_i$ consists of a single boundary
leaf. Therefore $T$ is a star. Conversely, for the star $K_{1,\ell}$ we have
$
\lambda_2=1=\frac{D}{\ell},
$
since $D=\ell$. Hence equality holds in~\eqref{eq:mainbound_intro}.
\end{proof}

\subsection{General \texorpdfstring{$\lambda_k$}{lambda\_k} bound}\label{subsec:lambdak}

\begin{proof}[Proof of Theorem~\ref{thm:lambdak}]
The bound $\lambda_k\le 1$ holds by Corollary~\ref{cor:leq1}.
Assume $\ell\ge 4Dk$. Let  $t=\lfloor\ell/(4Dk)\rfloor\ge 1$ and $M:=4k$.
We recursively construct $M$ pairwise disjoint connected components $C_1,\dots,C_M$.
Let $U_1:=T$, rooted at an arbitrary vertex. For the $m$-step ($1\le m\le M-1$), assume the remaining tree is $U_m$ and that the previously removed components $C_1, \dots, C_{m-1}$ each strictly contain fewer than $Dt$ boundary leaves. Under this inductive assumption, the number of boundary leaves in $U_m$ satisfies  
$$|\dO\cap U_m|>\ell-(m-1)Dt\ge
\ell-(M-2)Dt\ge
\ell-(4k-2)\frac{\ell}{4k}=
\frac{\ell}{2k}\ge 2Dt>Dt.$$
Hence the root of $U_m$ must have a child-subtree containing at least $t$
boundary leaves, so the greedy downward walk cannot stop at the root. Continuing downward while the current vertex has a child-subtree containing at least $t$ boundary
leaves, we eventually terminate at a vertex $u_m$ since the tree is finite.
Let $u_m$ be the vertex where the walk stops; equivalently,
the descendant subtree $(U_m)_{u_m}$ contains at least $t$ boundary leaves,
but every child-subtree of $u_m$ in $U_m$ contains fewer than $t$ boundary leaves.
Then $u_m$ is a proper descendant, so $(U_m)_{u_m}$ is attached to
$U_m\setminus (U_m)_{u_m}$ by exactly one edge, and the latter is connected.
Moreover, $$t\le |\dO\cap (U_m)_{u_m}|<Dt.$$
Indeed, if $u_m$ is a leaf then $|\dO\cap (U_m)_{u_m}|=1<Dt$.
Otherwise $u_m$ has at most $D-1$ children in $U_m$, and each child-subtree
contains fewer than $t$ boundary leaves. Let $$C_m:=(U_m)_{u_m},
\qquad
U_{m+1}:=U_m\setminus C_m.$$ After $M-1=4k-1$ steps, the remaining connected component $U_M$ satisfies $$|\dO\cap U_M|
\ge \ell-(M-1)Dt\ge\frac{\ell}{4k}\ge t.$$ Let $C_M:=U_M$. Thus we obtain a partition of $T$ into $M=4k$ disjoint connected components $C_1,\dots,C_M$, each containing at least $t$
boundary leaves.

Let $Q$ be a quotient graph obtained by contracting each $C_i$ to a vertex. Since $T$ is a tree, $Q$ is also a tree with $M$ vertices.
Let $n_i$ be the number of vertices in $Q$ with degree of $i$. Since the sum of degrees is $2M-2$, we have
$$
2M-2 = n_1 + 2n_2 + \sum_{i\ge 3} i n_i \ge n_1 + 2n_2 + 3(M - n_1 - n_2).
$$
Rearranging gives $2n_1+n_2\ge M+2$, which strictly implies $n_1+n_2>\frac{M}{2}$. Let $S$ be the set of vertices in $Q$ with degree at most $2$. We have $|S| \ge M/2 = 2k$.  
The subgraph $Q[S]$ induced by $S$ has maximum degree at most $2$, so it is a collection of paths and isolated vertices. In any acyclic graph with maximum degree at most $2$, the independence number is at least half the number of its vertices. Thus, $Q[S]$ contains an independent set $I$ of size at least $|S|/2 \ge k$. We select exactly $k$ components corresponding to the vertices in $I$. By definition, these $k$ components are mutually non-adjacent in $Q$, and each has degree at most $2$ in $Q$.

For each $j\in I$, define boundary data $g_j\equiv 1$ on $\dO\cap C_j$ and $g_j\equiv 0$ elsewhere. Define its extension $f_j \equiv 1$ on $C_j$ and $0$ elsewhere. Let $W=\mathrm{span}\{g_j\}_{j\in I}$ which has dimension $k$. For any  $g=\sum_{j\in I} a_j g_j \in W$, consider the extension $f=\sum_{j\in I} a_j f_j$. Since $I$ is an independent set, no two selected components share an edge. The Dirichlet energy of $f$ is exclusively supported on edges connecting selected components to unselected components. Since each selected $C_j$ has degree at most $2$ in $Q$, it is connected to at most $2$ unselected components. Thus, cross-terms completely vanish, yielding:
$$\mathcal{E}_T(f)=\sum_{j\in I}\deg_{C_j}(a_j-0)^2\le 2\sum_{j\in I}a_j^2.$$
The denominator of the Rayleigh quotient is $\sum_{j\in I} a_j^2 |\dO\cap C_j| \ge t \sum_{j\in I} a_j^2$. By Lemma~\ref{lem:harmmin}, the Steklov Rayleigh quotient is bounded by:
$$\mathcal{R}_{T,\dO}(g)\le \frac{\mathcal{E}_T(f)}{t\sum_{j\in I} a_j^2}\le\frac{2}{t}$$
Using $\lfloor x \rfloor \ge x/2$ for all $x \ge 1$, we have $1/t \le 8Dk/\ell$. Therefore, $\mathcal{R}_{T,\dO}(g) \le 16Dk/\ell$. By the min--max principle, $\lambda_k \le 16Dk/\ell$.
\end{proof}

\begin{remark}\label{rem:lambdak_vs_lambda2}
For $k=2$, Theorem~\ref{thm:lambdak} yields $\lambda_2\le 32D/\ell$ when $\ell\ge 8D$.
The sharp estimate in Theorem~\ref{thm:lambda2sharp} improves this to $\lambda_2\le D/\ell$ for all leaf--boundary trees.
\end{remark}

\begin{remark}[On the constant for higher eigenvalues]\label{rem:lambdak_sharp}
The factor $16$ in the bound $\lambda_k\le 16Dk/\ell$ arises from the $4k$--partition and independent set argument.
For $k=2$, the leaf--centroid method of Theorem~\ref{thm:lambda2sharp} achieves the sharp constant~$1$, but it does not generalize directly to $k\ge 3$ because one would need to partition the boundary into $k$ balanced parts with few separating edges.
It would be interesting to determine whether the constant $16$ can be reduced for specific values of $k$, or whether there exist families of trees showing that it is tight (see Problem~\ref{prob:lambdak}).
\end{remark}

\section{Numerical experiments}\label{sec:numerics}

This section reports systematic computational checks supporting Theorem~\ref{thm:lambda2sharp} and exploratory evidence for the extremal trees conjectured by Lin--Zhao~\cite{LinZhaoPlanar}.
All computations use the exact Dirichlet--to--Neumann matrix $\Lambda=L_{\dO\dO}-L_{\dO\Om}L_{\Om\Om}^{-1}L_{\Om\dO}$, and $\lambda_2$ is obtained as the second smallest eigenvalue of $\Lambda$.

\subsection{Exhaustive verification of the sharp bound}

We enumerate all unlabeled trees with $3\le n\le 16$ vertices and compute the ratio $\rho(T):=\lambda_2(T,\dO)|\dO|/\Delta(T)$, where $\dO$ is the leaf set and $\Delta(T)$ is the maximum degree.
Table~\ref{tab:enum} shows that $\rho(T)\le 1$ in every case, and that equality occurs only for stars.

\begin{table}[ht]
\centering
\begin{tabular}{r r r r}
\hline
$n$ & \# trees & $\max\rho$ & \# equality cases\\
\hline
3  & 1     & 1.000000 & 1 \\
4  & 2     & 1.000000 & 1 \\
5  & 3     & 1.000000 & 1 \\
6  & 6     & 1.000000 & 1 \\
7  & 11    & 1.000000 & 1 \\
8  & 23    & 1.000000 & 1 \\
9  & 47    & 1.000000 & 1 \\
10 & 106   & 1.000000 & 1 \\
11 & 235   & 1.000000 & 1 \\
12 & 551   & 1.000000 & 1 \\
13 & 1301  & 1.000000 & 1 \\
14 & 3159  & 1.000000 & 1 \\
15 & 7741  & 1.000000 & 1 \\
16 & 19320 & 1.000000 & 1 \\
\hline
\end{tabular}
\caption{Exhaustive verification of $\rho:=\lambda_2|\dO|/\Delta\le 1$ on all unlabeled trees with $3\le n\le 16$.
The unique equality case for each $n$ is the star $K_{1,n-1}$.}
\label{tab:enum}
\end{table}

\subsection{Exhaustive search in \texorpdfstring{$\mathrm{TS}(\ell,D)$}{TS(l,D)}}

Recall that $\mathrm{TS}(\ell,D)$ denotes the family of trees with $\ell$ leaves and maximum degree at most $D$.
Lin--Zhao~\cite[Definition~3.5]{LinZhaoPlanar} defined a balanced minimum--height candidate $T_b^*(\ell,D)$ and conjectured that it maximizes $\lambda_2$ for fixed $(\ell,D)$ when $\ell$ is sufficiently large.
Tables~\ref{tab:TS3} and~\ref{tab:TS4} report exhaustive searches for $D=3,4$ and $3\le\ell\le 10$.

\begin{table}[ht]
\centering
\begin{tabular}{r r r c r}
\hline
$\ell$ & $\max\lambda_2$ & $\lambda_2(T_b^*)$ & match & \# trees\\
\hline
3  & 1.000000 & 1.000000 & $\checkmark$ & 1 \\
4  & 0.500000 & 0.500000 & $\checkmark$ & 1 \\
5  & 0.333333 & 0.333333 & $\checkmark$ & 1 \\
6  & 0.333333 & 0.333333 & $\checkmark$ & 2 \\
7  & 0.221638 & 0.221638 & $\checkmark$ & 2 \\
8  & 0.200000 & 0.179806 & $\times$     & 4 \\
9  & 0.179806 & 0.179806 & $\checkmark$ & 6 \\
10 & 0.156047 & 0.156047 & $\checkmark$ & 11 \\
\hline
\end{tabular}
\caption{Exhaustive search over $\mathrm{TS}(\ell,3)$. The column ``match'' indicates whether $T_b^*(\ell,3)$ achieves the maximum $\lambda_2$; $\checkmark$ = yes, $\times$ = no.}
\label{tab:TS3}
\end{table}

\begin{table}[ht]
\centering
\begin{tabular}{r r r c r}
\hline
$\ell$ & $\max\lambda_2$ & $\lambda_2(T_b^*)$ & match & \# trees\\
\hline
3  & 1.000000 & 1.000000 & $\checkmark$ & 1 \\
4  & 1.000000 & 1.000000 & $\checkmark$ & 3 \\
5  & 0.454545 & 0.454545 & $\checkmark$ & 5 \\
6  & 0.400000 & 0.333333 & $\times$     & 20 \\
7  & 0.333333 & 0.333333 & $\checkmark$ & 65 \\
8  & 0.333333 & 0.333333 & $\checkmark$ & 276 \\
9  & 0.272727 & 0.272727 & $\checkmark$ & 1189 \\
10 & 0.250000 & 0.250000 & $\checkmark$ & 5604 \\
\hline
\end{tabular}
\caption{Exhaustive search over $\mathrm{TS}(\ell,4)$.}
\label{tab:TS4}
\end{table}

For $D=3$, the first mismatch appears at $\ell=8$: the maximum $\lambda_2=0.2$ is attained by a tree with a degree--$3$ vertex splitting leaves as $(4,2,2)$, whereas $T_b^*(8,3)$ has a $(3,3,2)$ split and yields $\lambda_2\approx 0.1798$.
For $D=4$, the mismatch appears at $\ell=6$: the optimum $\lambda_2=0.4$ occurs for a split $(3,1,1,1)$, while $T_b^*(6,4)$ yields $\lambda_2=1/3$.
These observations show that the hypothesis ``$\ell$ sufficiently large'' in the Lin--Zhao conjecture is essential.

\begin{observation}\label{obs:mismatch}
Consider $\mathrm{TS}(8,3)$.
We have $\lambda_2(T_1)=1/5$ for the tree $T_1$ with leaf split $(4,2,2)$ at its unique degree--$3$ vertex, and $\lambda_2(T_b^*(8,3))=\frac{7-\sqrt{17}}{16}\approx 0.1798$.
Thus the balanced minimum--height tree is not optimal for $\ell=8$, $D=3$.
The same phenomenon occurs in $\mathrm{TS}(6,4)$.
The optimum and the balanced candidate are illustrated in Figure~\ref{fig:TS83}.
\end{observation}

\begin{figure}[ht]
\centering
\begin{tikzpicture}\begin{scope}[
    ivert/.style={circle, draw, fill=black, inner sep=1.5pt},
    lvert/.style={circle, draw, fill=white, inner sep=2pt},
    edge from parent/.style={draw, thick},
    level distance=12mm,
    level 1/.style={sibling distance=20mm},
    level 2/.style={sibling distance=15mm},
    level 3/.style={sibling distance=10mm},
]
\node at (0,0.8) {$T_1$: split $(4,2,2)$};
\node[ivert] (r1) at (0,0) {}
    child { node[ivert] {} 
        child { node[ivert] {}
            child { node[lvert] {} }
            child { node[lvert] {} }
        }
        child { node[ivert] {}
            child { node[lvert] {} }
            child { node[lvert] {} }
        }
    }
    child { node[ivert] {}
        child { node[lvert] {} }
        child { node[lvert] {} }
    }
    child { node[ivert] {}
        child { node[lvert] {} }
        child { node[lvert] {} }  
    };
\node at (0,-4.2) {\footnotesize $\lambda_2=1/5=0.2$};
\end{scope}
\begin{scope}[
    ivert/.style={circle, draw, fill=black, inner sep=1.5pt},
    lvert/.style={circle, draw, fill=white, inner sep=2pt},
    edge from parent/.style={draw, thick},
    level distance=12mm,
    level 1/.style={sibling distance=20mm},
    level 2/.style={sibling distance=10mm},
]\node at (9,0.8) {$T_b^*(8,3)$: split $(3,3,2)$};
\node[ivert] (r2) at (9,0) {}
    child { node[ivert] {}
        child { node[ivert] {}
            child { node[lvert] {} }
            child { node[lvert] {} }  
        }
        child { node[lvert] {} }
    }
    child { node[ivert] {}
        child { node[ivert] {}
            child { node[lvert] {} }
            child { node[lvert] {} }  
        }
        child { node[lvert] {} }
    }
    child { node[ivert] {} 
        child { node[lvert] {} }
        child { node[lvert] {} }
    };
\node at (9,-4.2) {\footnotesize $\lambda_2=\frac{7-\sqrt{17}}{16}\approx 0.180$};\end{scope}
\end{tikzpicture}
\caption{The two key trees in $\mathrm{TS}(8,3)$. Interior vertices are filled (black); leaves (boundary) are open (white). The unbalanced tree $T_1$ (left) achieves a larger $\lambda_2$ than the balanced tree $T_b^*$ (right).}
\label{fig:TS83}
\end{figure}
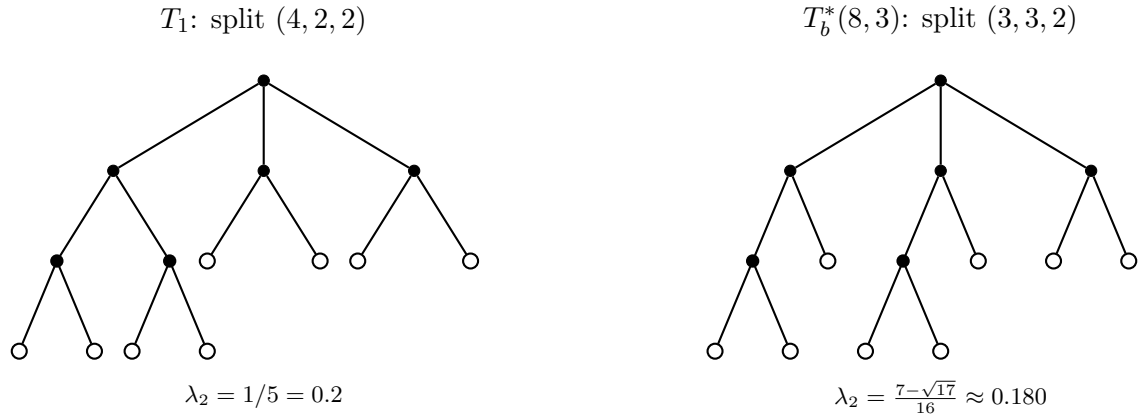

\subsection{Randomized experiments for larger leaf counts}

To complement the exhaustive regime, we generate random bounded--degree trees with a prescribed leaf count by recursively splitting leaf--counts.
Table~\ref{tab:random} records the best $\lambda_2$ observed among $40$ random samples for several $(\ell,D)$ and compares it to $\lambda_2(T_b^*(\ell,D))$.
In these trials, the balanced tree $T_b^*$ consistently has larger $\lambda_2$ than the best random sample, providing additional evidence for the Lin--Zhao conjecture beyond the range of exhaustive enumeration.

\begin{table}[ht]
\centering
\begin{tabular}{r r r r r r}
\hline
$D$ & $\ell$ & samples & max $\lambda_2$ (rand) & $\lambda_2(T_b^*)$ & max $\rho$ (rand)\\
\hline
3 & 20 & 40 & 0.068526 & 0.074610 & 0.457 \\
3 & 40 & 40 & 0.026111 & 0.037469 & 0.348 \\
4 & 20 & 40 & 0.091786 & 0.129844 & 0.459 \\
4 & 40 & 40 & 0.042099 & 0.066944 & 0.421 \\
5 & 20 & 40 & 0.111753 & 0.200000 & 0.447 \\
\hline
\end{tabular}
\caption{Randomized experiments.
The quantity $\rho:=\lambda_2\ell/D$ is bounded by $1$ in Theorem~\ref{thm:lambda2sharp}.}
\label{tab:random}
\end{table}

\paragraph{Reproducibility.}
All experiments were implemented in Python using NumPy for linear algebra and NetworkX for tree enumeration.
Code for reproducing the tables is available from the authors upon request.

\section{Concluding remarks}\label{sec:conclude}

We have proved that Steklov eigenvalues decrease monotonically under the edge--stretching operation. Moreover, we established a sharp upper bound $\lambda_2\le D/\ell$ for trees with leaf boundaries, derived explicit spectra for level-regular trees, and provided a general bound for higher eigenvalues $\lambda_k$.
These results fit into the extremal program for $\lambda_2$ on trees with constraints on leaf count and maximum degree:
\begin{itemize}[leftmargin=2em]
\item Corollary~\ref{cor:deg2} reduces the search for extremizers to homeomorphically irreducible trees.
\item Theorem~\ref{thm:levelspec} provides a closed formula for an important candidate family, allowing an exact comparison.
\item Theorems~\ref{thm:lambda2sharp} and~\ref{thm:lambdak} give upper benchmarks that any candidate must respect.
\end{itemize}
A natural next step is to combine these tools with structural transformations that move leaves between branches while controlling the DtN form, in order to approach conjectured extremizers such as balanced minimum--height trees~\cite{LinZhaoPlanar}.

It is natural to ask similar questions on other graph families.
Monotonicity of Steklov eigenvalues with respect to the addition of edges (see~\cite[Theorem~3.1]{LinZhaoPlanar}) justifies the consideration of the problem on subgraph--closed graph families.

\begin{problem}\label{prob:TS}
Let $\mathrm{TS}(\ell,D)$ be the set of all trees with $\ell$ leaves and maximum degree at most $D$.
Determine the trees with maximum $\lambda_2$ among $\mathrm{TS}(\ell,D)$.
\end{problem}

\begin{problem}\label{prob:lambdak}
Determine the sharp constant in the inequality $\lambda_k\le c_k^*\,Dk/\ell$ for $k\ge 3$.
\end{problem}

\end{document}